\documentclass[11pt]{article}
\usepackage[english]{babel}
\usepackage[a4paper,left=1.9cm,right=1.9cm,top=2.5cm,bottom=2.5cm]{geometry}
\usepackage{graphicx}
\usepackage[pdfborder={0 0 0}]{hyperref}
\usepackage[utf8]{inputenc}

\usepackage{authblk}
\usepackage{mathtools}
\usepackage{amssymb}
\usepackage{dsfont}
\usepackage{amsmath}
\usepackage{url}
\usepackage{float}
\usepackage{comment}
\usepackage[T1]{fontenc}
\usepackage{bm}
\usepackage{graphicx}
\usepackage[normalem]{ulem}
\usepackage{commath}
\usepackage{amsthm}
\usepackage{mathrsfs}  
\usepackage{amsmath,amscd}
\usepackage{hyperref}
\usepackage{bm}
\usepackage{array}
\usepackage[nottoc,notlot,notlof]{tocbibind}
\usepackage[title]{appendix}
\usepackage{tikz}
\usepackage[title]{appendix}
\usepackage{xcolor}
\usepackage{leftidx}

\usetikzlibrary{fadings}

\usepackage{pgfplots}

\usepackage{cleveref}

\usepackage{tikz-cd}
\usetikzlibrary{matrix,arrows}
\usepackage{subfiles}

\renewcommand{\subset}{\subseteq}

\newcommand{\End}{\mathrm{End}}

\newcommand{\dd}{\textup{d}}

\newcommand{\mc}{\mathcal}
\newcommand{\mf}{\mathfrak}
\newcommand{\mb}{\mathbb}

\makeatletter
\newcommand{\opnorm}{\@ifstar\@opnorms\@opnorm}

\makeatother

\pgfplotsset{compat=1.10}
\usepgfplotslibrary{fillbetween}
\usetikzlibrary{patterns}

\newcommand{\Addresses}{{
\bigskip
\footnotesize

\textsc{Max Planck Institute for Mathematics, Vivatsgasse 7, 53111 Bonn, Germany} \par \nopagebreak
    \textit{E-mail address}:\href{mailto:ksingh@mpim-bonn.mpg.de}{ksingh@mpim-bonn.mpg.de}
}}

\theoremstyle{plain}
\newtheorem{thm}{Theorem}[section]
\newtheorem*{thm*}{Theorem}
\newtheorem{lem}[thm]{Lemma}

\theoremstyle{definition}
\newtheorem{defn}[thm]{Definition} 
\newtheorem{exmp}[thm]{Example}
\newtheorem{rmk}[thm]{Remark}

\newtheorem*{conv*}{Convention}
\newtheorem*{prob*}{Problem}

\newtheorem{assumptions}[thm]{Assumptions}

\title{A unified approach to some rigidity and stability problems in algebra}
\author{Karandeep J. Singh }
\date{\null}

\begin{document}

\maketitle
\begin{abstract}
\end{abstract}
In this note, we use give some algebraic applications of a  previous result by the author which compares the deformations parameterized by the Maurer-Cartan elements of a differential graded Lie algebra, and a differential graded Lie subalgebra: It gives a criterion for the map on the space of Maurer-Cartan elements up to gauge equivalence, induced by the inclusion of the subalgebra, to be locally surjective. By making appropriate choices for the differential graded Lie algebra and a differential graded Lie subalgebra, we recover some classical results in the deformation theory of finite-dimensional Lie and associative algebras. We consider rigidity of Lie and associative algebras, Lie algebra morphisms, and give a quick proof of the fact that the deformation theories of a unital associative algebra as a unital associative algebra and as an associative algebra are equivalent. We then turn to stability of Lie subalgebras and their morphisms under deformations of the ambient Lie algebra structure, and study when the compatibility with a geometric structure of a representation of a Lie algebra on a vector space is stable under deformations of the representation.
\tableofcontents
\section{Introduction}
Throughout mathematics, deformation problems appear in various forms (see e.g. \cite{KodairaSpencer,gerstenhaber,nijenhuislie}). A principle postulated by Deligne \cite{delignedgla} states that deformation problems are often encoded in an algebraic structure called a \emph{differential graded Lie algebra}: the structures of interest are parameterized by special elements called \emph{Maurer-Cartan elements}, while isomorphisms of structures are captured by an equivalence relation on the Maurer-Cartan elements, which can be described intrinsic to the differential graded Lie algebra. As such, studying algebraic properties of differential graded Lie algebras can lead to insights to the deformation problem they encode.

An important example of this phenomenon is the relation between cohomological properties of the differential graded Lie algebras and local properties of the space of structures of interest. The goal of this note is to provide some algebraic applications of the author's result \cite{KarandeepStability} which provides such a relation. These applications recover several classical results in the deformation theory of Lie and associative algebras, and allow for the result to be applied with less technicalities than the geometric applications in \cite{KarandeepStability}. We proved the following result, which we state here in simplified form.
\begin{thm}\label{thm:mainthmbas}
Let $(\mf g, \partial,[-,-])$ be a differential graded Lie algebra, and let $\mf h \subset \mf g$ be a differential graded Lie subalgebra of finite codimension. Let $Q\in \mf h^1$ be a Maurer-Cartan element of $\mf h$, and assume that 
$$
H^1(\mf g/\mf h, \overline{\partial + [Q,-]}) = 0.
$$
Then every Maurer-Cartan element $Q'\in \mf g^1$ near $Q$ is equivalent to an element of $\mf h^1$.
\end{thm}

This theorem should be thought of as a statement about infinitesimal stability implying local stability: as $H^1(\mf h, \overline{\partial + [Q,-]})$ and $H^1(\mf g,\overline{\partial + [Q,-]})$ can be viewed as the tangent spaces to the spaces of Maurer-Cartan elements up to equivalence of $\mf h$ and $\mf g$ at the class of $[Q]$ respectively, the hypothesis implies that the map
$$
H^1(\mf h,\overline{\partial + [Q,-]}) \to H^1(\mf g, \overline{\partial + [Q,-]})
$$
is surjective, while the conclusion states that the induced map on Maurer-Cartan elements up to equivalence is locally surjective around $Q$. When the subalgebra $\mf h$ is defined by a linear property $P$, Theorem \ref{thm:mainthm} can be interpreted as a criterion for when every object near $Q$ has the property $P$ up to equivalence. We will further clarify this idea by studying some examples in the deformation theory of Lie and associative algebras.

A geometric approach to pass from cohomological to local properties for deformation problems in Lie theory was given in \cite{crainic2013survey}: and Theorem \ref{thm:lierigid}, Theorem \ref{thm:rigmorph}, Theorem \ref{thm:stabsubalg} and a weaker version of Theorem \ref{thm:stabmorph2} were also proven there.
\subsection*{Outline of the paper}
The outline of the paper is as follows:
\begin{itemize}
    \item[-] In Section \ref{sec:defthydgla}, we recall the basics of differential graded Lie algebras and their role in deformation theory, and state the full version of Theorem \ref{thm:mainthmbas} in Section \ref{sssec:mainthm}   (Theorem \ref{thm:mainthm}).
    \item[-] In Section \ref{sec:applications}, we give some applications of Theorem \ref{thm:mainthm}. 
    \begin{itemize}
        \item[-] In Section \ref{sssec:lierigid}, we recover the rigidity result of \cite{nijenhuislie} for finite-dimensional Lie algebras (Theorem \ref{thm:lierigid}).
        \item[-] In Section \ref{sssec:assrigid}, we recover the rigidity result of \cite{gerstenhaber} for finite-dimensional associative algebras (Theorem \ref{thm:assrigid}).
        \item[-] In Section \ref{sssec:unitvnon}, we use Theorem \ref{thm:mainthm} to give a quick proof of the well-known result which states that the deformation theories of a unital associative algebra as a unital associative algebra and as an associative algebra are equivalent (Theorem \ref{thm:unitvsnon}).
        \item[-] In Section \ref{sec:rigmorphism}, we recover the rigidity result of \cite{cohomgradlie} for morphisms of finite-dimensional Lie algebras (Theorem \ref{thm:rigmorph}). 
        \item[-] In Section \ref{ssec:stabsubalg}, we recover the stability result for Lie subalgebras of a given Lie algebra \cite{stabsubalgmorph} (Theorem \ref{thm:stabsubalg}).
        \item[-] In Section \ref{ssec:stabmorph}, we give two approaches to obtain results on the stability of morphisms of Lie algebras (Theorem \ref{thm:stabmorph1} and Theorem \ref{thm:stabmorph2}), the latter of which agrees with \cite{stabsubalgmorph}, but our approach guarantees stability under a larger class of deformations.
        \item[-] In Section \ref{ssec:mapintosubalg}, we apply Theorem \ref{thm:mainthm} to obtain a criterion for when the property of a Lie algebra morphism mapping into a fixed subalgebra of the codomain is stable (Theorem \ref{thm:morphsubalgconcr}). As applications, we give criteria for the the compatibility of a representation of a Lie algebra with a geometric structure (Example \ref{ex:geostr}).
    \end{itemize}
    
\end{itemize}
\paragraph{\bf Acknowledgements} We would like to thank Marco Zambon for fruitful discussions and helpful comments. We would like to thank the Max Planck Institute for Mathematics for its hospitality and financial support, where a part of this work was done. We also thank KU Leuven, and acknowledge the FWO and FNRS under EOS projects G0H4518N and G0I2222N. 
\section{Differential graded Lie algebras, stability and rigidity}\label{sec:defthydgla}
In this section, we give some background on differential graded Lie algebras and a way they are used in deformation theory. 
\subsection{Differential graded Lie algebras}
\begin{defn}\label{def:dgla}
    A \emph{differential graded Lie algebra} is a triple $(\mf g,\partial, [-,-])$, where
    \begin{itemize}
        \item [i)] $\mf g = \bigoplus_{i\in \mb Z} \mf g^i$ is a graded vector space,
        \item [ii)] $\partial:\mf g\to \mf g$ is map of degree 1,
        \item [iii)] $[-,-]:\mf g \times \mf g \to \mf g$ is a graded skew-symmetric map, 
    \end{itemize}
    such that
    \begin{itemize}
        \item [a)] $\partial^2 = 0$,
        \item [b)] for all $x, y\in \mf g$ homogeneous, $$\partial[x,y] = [\partial x,y] +(-1)^{|x|} [x,\partial y],$$
        \item [c)] for all $x,y,z \in \mf g$ homogeneous, $$[[x,y],z] = [x,[y,z]] - (-1)^{|x||y|}[y,[x,z]].$$
    \end{itemize}
\end{defn}
It was postulated by Deligne \cite{delignedgla} that deformation problems are governed by differential graded Lie algebras, or more generally, $L_\infty$-algebras. The main idea is that equations that are satisfied by structures, can often be realized as an equation involving the structure maps of a differential graded Lie algebra called the \emph{Maurer-Cartan equation}.
\begin{defn}
    Let $(\mf g, \partial, [-,-])$ be a differential graded Lie algebra, and let $Q\in \mf g^1$. Then $Q$ is a \emph{Maurer-Cartan} element if 
    \begin{equation}\label{eq:MC}
        \partial(Q) + \frac{1}{2}[Q,Q] = 0.
    \end{equation}
\end{defn}
Note that $\mf g^0$ is an ordinary Lie algebra. Moreover, there is a natural action of $\mf g^ 0$ on $\mf g^1$ by symmetries, which encode equivalences of deformations (see e.g. \cite{mandef2}). To define it, we need to equip $\mf g^1$ with a topology. As $\mf g^ 0$ and $\mf g^1$ will be finite-dimensional in all applications of this paper, we will restrict to this case.
\begin{defn}\label{def:gaugeeq}
    Let $X\in \mf g^0$, $Q\in \mf g^1$. Define $Q^X$ as the time-1 solution of the initial value problem
    \begin{equation}\label{eq:gaugeeq}
        \frac{d}{dt}Q_t = \partial(X) - [X,Q_t], \,\,\, Q_0 = Q.
    \end{equation}
\end{defn}
\begin{rmk}
    The relation $Q\sim Q^X$ generates an equivalence relation on $\mf g^1$: $Q$ is said to be \emph{gauge equivalent} to $Q'$ if they can be connected by a finite concatenation of paths as in \eqref{eq:gaugeeq}.
\end{rmk}
When $\mf g$ is degreewise finite-dimensional, we can give an explicit expression for $Q^X$. First observe that due to property c) in Definition \ref{def:dgla}, 
$$
[-,-]:\mf g^0\times \mf g^1 \to \mf g^1
$$
makes $\mf g^1$ into an $\mf g^0$-module.
\begin{lem}\label{lem:gaugeactionexpl} Let $Q\in \mf g^1$, $X\in \mf g^0$. Let $G$ be the 1-connected Lie group integrating $\mf g^0$ and let $\exp:\mf g^0 \to G$ be the exponential map. Denote by $\rho:G\to GL(\mf g^1)$ the induced $G$-representation on $\mf g^1$. Then the solution to the initial value problem \eqref{eq:gaugeeq} is given by
$$
Q_t = \rho(\exp(-tX))(Q) + \int_{0}^t \rho(\exp(-sX))(\partial X) \dd s.
$$    
Consequently, 
\begin{equation}\label{eq:gaugeaction}
    Q^X = \rho(\exp(-X))(Q) + \int_0^1 \rho(\exp(-sX))(\partial X)\dd s.
\end{equation}
\end{lem}
\begin{rmk} Observe that due to the linearity of \eqref{eq:gaugeeq} in $X$, we have $Q_t = Q^{tX}$ for all $t \in \mb R$. That is, the time-$t$ solution associated to $X$ is the time-$1$ solution associated to $tX$.
\end{rmk}
\subsection{Differential graded Lie algebras in deformation theory}
\subsubsection{Rigidity}
In mathematics, rigidity questions arise naturally:
\begin{equation}\label{qrig}\text{Given a structure $Q_0$, when is every nearby structure $Q$ isomorphic to $Q_0$?}\end{equation}
Without any additional information, this is not a well-defined question, as the notions of closeness and isomorphisms are not defined. However, in many examples there is a natural meaning, which is captured in a differential graded Lie algebra $(\mf g, \partial, [-,-])$. The idea is that in this differential graded Lie algebra, we have the correspondences 
\begin{align*}
    \{\text{Structures near $Q_0$}\} &\longleftrightarrow  \{\text{Maurer-Cartan elements of $(\mf g, \partial, [-,-])$ near $0$}\},\\
    \{\text{Isomorphisms of structures}\} &\longleftrightarrow \{\text{Gauge equivalences}\}.
\end{align*}
Then the rigidity question is equivalent to: 
\begin{center}
    When is every Maurer-Cartan element near $0$ gauge equivalent to $0$?
\end{center}
Concretely, we are interested in when the map
$$
\text{ev}_0:\mf g^0 \to \mf g^1,\,\,\, X \mapsto 0^X
$$
is locally surjective onto a neighborhood of the zero set of the map
$$
MC:\mf g^1 \to \mf g^2, \,\,\, Q \mapsto \partial(Q) + \frac{1}{2}[Q,Q].
$$
We turn to the infinitesimal side, and analyse when surjectivity holds at the level of tangent spaces, which we define as if the zero set of $MC$ is a smooth manifold. Differentiating $\text{ev}_0$ at $0\in \mf g^0$, we find
$$
d(\text{ev}_0)_0(v) = \left.\frac{d}{dt}\right|_{t = 0} 0^{tv} = \partial(v) - [v,Q_0] = \partial(v).
$$
Moreover, the derivative of $MC$ at $0\in \mf g^1$ is $\partial: \mf g^1 \to \mf g^2$. Consequently, $0$ is ``infinitesimally rigid'' if $\partial(\mf g^0) = \ker(\partial:\mf g^1\to \mf g^2)$. In other words, if
$$
H^1(\mf g,\partial) = 0.
$$
It turns out that this condition is sufficient for rigidity if $\mf g$ is degreewise finite-dimensional. Note that $H^1(\mf g,\partial)$ can be interpreted as the ``tangent space'' to $[0]\in \frac{MC(\mf g,\partial,[-,-])}{\sim_{\mf g}}$. For the ``tangent space'' to the class of a general $Q\in MC(\mf g,\partial,[-,-])$, we similarly obtain $H^1(\mf g, \partial + [Q,-])$.
\subsubsection{Stability}\label{sec:stability}
Rigidity is quite a strong property, and often fails. In this subsection, we explain the weaker notion of stability of a property. Geometrically, rigidity of a Maurer-Cartan element $Q$ means that the inclusion of a point
$$
\{\ast\} \hookrightarrow \frac{MC(\mf g, \partial,[-,-])}{\sim}, \,\,\, \ast \mapsto [Q]
$$
is locally surjective. It is natural to consider the inclusion of other nice subsets of $\frac{MC(\mf g,\partial,[-,-])}{\sim}$, and study when the inclusion is locally surjective. If the Maurer-Cartan elements of $(\mf g,\partial, [-,-])$ correspond to certain algebraic or geometric structures, it is natural from a deformation theoretic perspective to consider subsets defined by imposing a property on the Maurer-Cartan elements and considering the orbits of the objects with the property. It turns out that the properties can often be encoded inside a differential graded Lie subalgebra $\mf h \subset \mf g$. We then have the correspondences:
\begin{align*}
    \{\text{Structures near $Q_0$}\} &\longleftrightarrow  \{\text{Maurer-Cartan elements of $(\mf g, \partial, [-,-])$ near $0$}\},\\
    \{\text{Structures near $Q_0$ with property $P$}\} & \longleftrightarrow \{\text{Maurer-Cartan elements of $\mf h\subset \mf g$ near $0$}\},\\
    \{\text{Isomorphisms of structures}\} &\longleftrightarrow \{\text{Gauge equivalences}\}.
\end{align*}

Then the subset of structures with property $P$ up to $\mf g$-equivalence can be realized as the image of the natural map
\begin{equation}\label{eq:locsurjstab}
\frac{MC(\mf h,\partial,[-,-])}{\sim_{\mf h}} \to \frac{MC(\mf g, \partial, [-,-])}{\sim_{\mf g}},
\end{equation}
where $\mf h \subset \mf g$ is a differential graded Lie subalgebra, and $\sim_{\mf h}$ is the induced equivalence on $\mf h^1$. Note however, that the map need not be injective, as in general, more things are identified on the right hand side. We again consider the infinitesimal version. As in the previous section, local surjectivity around $0 \in \mf g$ at the infinitesimal level means that the induced map
\begin{equation}\label{eq:infsurjstab}
H^1(\mf h,\partial) \to H^1(\mf g,\partial)
\end{equation}
is surjective. By the long exact sequence in cohomology associated to the short exact sequence $$\mf h \hookrightarrow \mf g\twoheadrightarrow \mf g/\mf h,$$ the surjectivity of \eqref{eq:infsurjstab} is equivalent to the vanishing of the map
$$
H^1(\mf g,\partial) \to H^1(\mf g/\mf h,\overline{\partial }),
$$
where $\overline{\partial}$ is the induced differential on the quotient $\mf g/\mf h$. It turns out that the stronger condition $H^1(\mf g/\mf h,\overline{\partial}) = 0$ implies that the map \eqref{eq:locsurjstab} is locally surjective, which we discuss in the next section.
\begin{rmk} \label{eq:gaugevsiso}
    Often the gauge equivalences are not in bijection with the isomorphisms, but only correspond to a subset of isomorphisms, as we will see in examples. This has two sides: while vanishing of the relevant cohomology group gives more information than the existence of some isomorphism, it does not necessarily see all isomorphisms. A structure may therefore be rigid, while the rigidity is not seen by the differential graded Lie algebra.
\end{rmk}
\subsubsection{A stability result in terms of differential graded Lie algebras}\label{sssec:mainthm}
In this subsection, we recall \cite[Theorem 3.20]{KarandeepStability}, and show that in the case of degreewise finite-dimensional differential graded Lie algebras, all technical assumptions are satisfied. 

We summarise the data and assumptions needed for \cite[Theorem 3.20]{KarandeepStability}.
\begin{assumptions} \label{ass:mainthm} Assume that we have the following:
    \begin{itemize}
        \item[i)] A differential graded Lie algebra $(\mf g, \partial, [-,-])$ such that $\mf g^i$ is a locally convex vector space for $i = 0,1,2$,
        \item[ii)] a differential graded Lie subalgebra $\mf h \subset \mf g$ of degree-wise finite codimension, for which $\mf h^i \subset \mf g^i$ is a closed subspace,
        \item[iii)] splittings $\sigma_i:\mf g^i/\mf h^i\to \mf g^i$ for $i = 0,1$.
        \item[iv)] a Maurer-Cartan element $Q\in \mf h^1$,
    \end{itemize}
    satisfying the following data:
    \begin{itemize}
        \item[a)] $\partial:\mf g^1 \to \mf g^2$ and $[-,-]:\mf g^1 \times \mf g^1 \to \mf g^2$ are continuous.
        \item[b)] There is an open neighborhood $U\subset \mf g^0/\mf h^0$ of $0\in \mf g^0/\mf h^0$, such that for all $Q' \in \mf g^1$, the gauge action as in Definition \ref{def:gaugeeq} of $\sigma_0(v)$ for $v\in U$ on $Q'$ is defined, the assignment
        $$
        U\times \mf g^1\ni (v,Q') \mapsto (Q')^{\sigma_0(v)}\in \mf g^1
        $$
        is jointly continuous, and for each fixed $Q'$, the class $(Q')^{\sigma_0(v)}$ mod $\mf h^1$ depends smoothly on $v\in U$.
        \item[c)] For $v\in U$, $Q'\in \mf g^1$ is Maurer-Cartan if and only if $(Q')^{\sigma_0(v)}$ is Maurer-Cartan.
    \end{itemize}
\end{assumptions}
Under these assumptions, we have the following result.
\begin{thm}\label{thm:mainthm}
    Assume that we are in the setting as described in Assumptions \ref{ass:mainthm}. Assume that 
    \begin{equation}\label{eq:cohomcond}
    H^1(\mf g/\mf h, \overline{\partial + [Q,-]}) = 0.
    \end{equation}
    Then for every open neighborhood $V$ of $0 \in U$, there exists an open neighborhood $\mc U\subset MC(\mf g)$ of $Q$ such that for any $Q'\in \mc U$ there exists a family $I\subset V$ smoothly parametrized by an open neighborhood of $0\in \ker(\overline{\partial:\mf g^0/\mf h^0 \to \mf g^1/\mf h^1})$ with $(Q')^{\sigma_0(v)}\in \mf h^1$ for $v \in I$.
\end{thm}
It turns out that when $\mf g$ is degree-wise finite-dimensional, for any choice of splittings $\sigma_0$ and $\sigma_1$, properties a), b) and c) above are automatically satisifed for $U = \mf g^0/\mf h^0$.
\begin{lem}
    Let $(\mf g, \partial,[-,-])$ be a degree-wise finite-dimensional differential graded Lie algebra. Then for any differential graded Lie subalgebra $\mf h\subset \mf g$, and any choice of splittings $\sigma_i:\mf g^i/\mf h^i \to \mf g^i$, Assumptions \ref{ass:mainthm}a)-c) are satisfied.
\end{lem}
\begin{proof}
    Note that $\mf h\subset \mf g$ is trivially a closed subspace of finite codimension with respect to the usual topology on finite-dimensional vector spaces. Let $\sigma_i:\mf g^i/\mf h^i\to \mf g^i$ be splittings. We check properties a)-c).
    \begin{itemize}
        \item [a)] Multilinear maps between finite-dimensional vector spaces are automatically continuous.
        \item [b)] This follows directly from the expression given in Lemma \ref{lem:gaugeactionexpl}.
        \item [c)] Although this could be checked directly, using the expression given in \eqref{eq:gaugeaction}, we use the defining differential equation. Let $Q_t$ be the solution to \eqref{eq:gaugeeq} associated to $X\in \mf g^0$ such that $Q_0$ is a Maurer-Cartan element. Define 
        $$
        \alpha_t := \partial(Q_t) + \frac{1}{2}[Q_t,Q_t] \in \mf g^2.
        $$
        Then it can be shown that
        $$
        \frac{d}{dt}\alpha_t = -[X,\alpha_t], \,\,\, \alpha_0 = 0 \in \mf g^2.
        $$
        Observe that the constant path $\alpha_t = 0$ satisfies the above equation. As this is a first order linear homogeneous ordinary differential equation, it has a unique solution. Consequently, $\alpha_t \equiv 0$, and $Q^X$ is Maurer-Cartan.
    \end{itemize}
\end{proof}
\section{Applications}\label{sec:applications}
In this section we recover various stability and rigidity results by making appropriate choices for the differential graded Lie algebras $(\mf g,\partial,[-,-])$ and subalgebra $\mf h\subset \mf g$.
All vector spaces are finite-dimensional over either the real or complex numbers and maps will be $\mb R$-linear or $\mb C$-linear respectively.

As the obstruction \eqref{eq:cohomcond} is some Lie algebra cohomology for most examples, we recall the definition.
\begin{defn}
    Let $(V,\mu)$ be a Lie algebra, and let $\rho:V\to \End(W)$ be a representation. The \emph{Chevalley-Eilenberg complex of $(V,\mu)$ with values in $W$} is defined by
    $$
    C_{CE}^\bullet(V,W):=(\wedge^\bullet V^\ast \otimes W, d_{CE}^\rho),
    $$
    where for $\alpha:\wedge^k V\to W$, $x_0,\dots,x_k\in V$, we have
    \begin{align*}
    d_{CE}^\rho(\alpha)(x_0,\dots, x_k) := \sum_{i=0}^k (-1)^i \rho(x_i)&(\alpha(x_0,\dots, \widehat{x_i}, \dots, x_k)) \\&- \sum_{0\leq i<j\leq k} (-1)^{i+j}\alpha([x_i,x_j], x_0, \dots, \widehat{x_i},\dots, \widehat{x_j},\dots, x_k).
    \end{align*}
\end{defn}
\begin{exmp} We provide some examples in low degree:
    \begin{itemize}
        \item [-] In degree $0$, the cocycles are the invariants of the representation $\rho$.
        \item[-] In degree $1$, the cocycles are the $W$-valued derivations on $(V,\mu)$, which we will denote by $\text{Der}((V,\mu),W)$. When $W = V$ and $\rho$ is the adjoint representation, we suppress $W$ from the notation. The coboundaries are the derivations of the form $\delta_w:V\to W$ for $w\in W$, given by $\delta_w(v):= \rho(v)(w)$ for $v\in V$. 
    \end{itemize}
\end{exmp}
\subsection{Rigidity of Lie and associative algebras}\label{ssec:rigidlieass}
When dealing with rigidity from the perspective of Theorem \ref{thm:mainthm}, the property of a structure $Q_0$ we wish to preserve up to equivalence is the property of being equal to $Q_0$. Hence, that corresponds to taking $\mf h = 0$.
\subsubsection{Lie algebras}\label{sssec:lierigid}
Let $V$ be a vector space. A Lie algebra structure on $V$ is a linear map $\mu:\wedge^2 V \to V$, satisfying the Jacobi identity: for $x,y,z \in V$, we have
$$
\mu(x,\mu(y,z)) + \mu(y,\mu(z,x)) + \mu(z,\mu(x,y)) = 0.
$$
It turns out that there is a differential graded Lie algebra $(\mf g,0, [-,-])$ such that its Maurer-Cartan elements are precisely the Lie algebra structures on $V$:
\begin{lem}[\cite{cohomgradlie}] Let $V$ be a vector space. The graded vector space $\mf g^\bullet:= \wedge^{\bullet+1} V^\ast \otimes V$ carries a graded Lie algebra structure which for $\alpha \otimes v \in \mf g^k, \beta \otimes w \in \mf g^l$ is defined by
$$
[\alpha \otimes v,\beta \otimes w]_{NR} = \beta \wedge \iota_w(\alpha) \otimes v - (-1)^{kl} \alpha\wedge \iota_v(\beta) \otimes w.
$$
An element $\mu \in \mf g^1$ is Maurer-Cartan, i.e. $[\mu,\mu]_{NR} = 0$, if and only if $\mu$ satisfies the Jacobi identity. Moreover, for $A \in \mf g^0 = \End(V), \mu \in \mf g^1$, we have 
$$
\mu^A = \exp(A)\circ \mu(\exp(-A)-, \exp(-A)-).
$$
\end{lem}
Now that Lie brackets on a vector space are realized as Maurer-Cartan elements of a degreewise finite-dimensional graded Lie algebra, we are in the setting of Section \ref{sec:defthydgla}. Moreover, as the gauge equivalences are given by the action of (the connected component of the identity) of $GL(V)$, the classical rigidity question is an instance of the question \eqref{qrig}. Recall that a choice of Lie algebra structure $\mu\in \mf g^1$ induces a differential on $\mf g$ by $[\mu,-]_{NR}$. We can give an explicit description of the cohomology.
\begin{lem}
    The differential $[\mu,-]_{NR}$ is equal to the Chevalley-Eilenberg differential of $(V,\mu)$ with values in the adjoint representation $V$. 
\end{lem}
Consequently, we obtain:
\begin{thm}[\cite{nijenhuislie}] \label{thm:lierigid}
Let $(V,\mu)$ be a Lie algebra.  If $$H^2_{CE}(V,V) = 0,$$ where $V$ acts on itself by the adjoint representation, $\mu$ is rigid. Moreover, for Lie algebra structures $\mu'$ near $\mu$, the endomorphisms $A:V\to V$ for which $\exp(A)$ maps $\mu'$ to $\mu$ are smoothly parametrized by an open neighborhood of $0\in \text{Der}(V,\mu)$.
\end{thm}
\subsubsection{Associative algebras}\label{sssec:assrigid}
For associative algebras, the situation is very similar to Lie algebras. An associative algebra structure on a vector space $V$ is a linear map $m:V\otimes V \to V$, which is associative: for $x,y,z \in V$, we have
$$
m(m(x\otimes y)\otimes z) = m(x\otimes m(y\otimes z)).
$$
For now, we do not address possible unitality of the algebra.

Again, there is a differential graded Lie algebra $(\mf g,0,[-,-])$ such that its Maurer-Cartan elements are precisely the associative algebra structures on $V$.
\begin{lem}[\cite{cohomassring}]
    Let $V$ be a vector space. The graded vector space $\mf g^\bullet := (V^\ast)^{\otimes (\bullet+1)} \otimes V$ carries a graded Lie algebra structure which for $\alpha \otimes v \in \mf g^k$, $\beta\otimes w\in \mf g^l$ is defined by
    $$
    [\alpha \otimes v, \beta \otimes w]_G = \beta\otimes \iota_w(\alpha) \otimes v -(-1)^{kl} \alpha\otimes \iota_v(\beta)\otimes w.
    $$
    An element $m \in \mf g^1$ is Maurer-Cartan, i.e. $[m,m]_G = 0$, if and only if the multiplication $m$ is associative. Moreover, for $A \in \mf g^0 = \End(V), m\in \mf g^1$, we have $$m^A=   \exp(A) \circ m(\exp(-A)-, \exp(-A)-). $$
\end{lem}

Again, the gauge equivalences are given by the action of (the connected component of the identity) of $GL(V)$. The Maurer-Cartan element $m\in \mf g^1$ induces a differential $[m,-]_G$ on $\mf g$. We can give an explicit description of this differential:
\begin{lem}
The differential $[m,-]_G$ is equal to the Hochschild differential of $(V,m)$ with values in the $(V,m)$-bimodule $V$.
\end{lem}
Consequently, Theorem \ref{thm:mainthm} implies:
\begin{thm}[\cite{gerstenhaber}]\label{thm:assrigid}
    Let $(V,m)$ be an associative algebra. If $$H^2_{H}(V,V) = 0,$$ where $H^2_H(V,V)$ is the Hochschild cohomology of $(V,m)$ with values in the bimodule $V$, then $m$ is rigid. Moreover, for associative algebra structures $m'$ near $m$, the endomorphisms $A:V\to V$ for which $\exp(A)$ maps $m'$ to $m$ are smoothly parametrized by an open neighborhood of $0\in \text{Der}(V,m)$.
\end{thm}

\subsubsection{Unital algebras}\label{sssec:unitvnon}
We now turn our attention to the unital case. However, instead of giving a direct approach to the problem, we use Theorem \ref{thm:mainthm} to show that the natural map 
\begin{equation}\label{eq:inclunital}
\frac{\{\text{Associative algebra structures on $V$ with unit $1\in V$}\}}{\{\text{Linear isomorphisms preserving $1\in V   $}\}_0} \to \frac{\{\text{Associative algebra structures on $V$}\}}{\{\text{Linear isomorphisms}\}_0}
\end{equation}
induces a local bijection, where the subscript $0$ denotes the connected component of the identity.

Let $(V,m)$ be an associative algebra with unit $1\in V$. We show that the left hand side can be modelled by a differential graded Lie subalgebra of $(\mf g^\bullet =(V^\ast)^{\otimes(\bullet+1)} \otimes V,[m,-]_G,[-,-]_G)$.
\begin{lem}
    Let $\mf h^\bullet :=  ((V/\mb K1)^\ast)^{\otimes (\bullet+1)} \otimes V$, where $\mb K\in \{\mb R, \mb C\}$. Then $\mf h^\bullet$ is a dg-Lie subalgebra of $\mf g^\bullet$, using the pullback map along $p:V \to V/\mb K1$, with Maurer-Cartan elements 
    $$
    m':V\otimes V \to V
    $$
    such that
    \begin{itemize}
        \item[i)] $m+m':V\otimes V \to V$ is an associative algebra structure,
        \item[ii)] $m'(1,-) = m'(-,1) = 0$, that is, $1\in V$ is a unit for $m+m'$.
    \end{itemize}
\end{lem}

Next, we show that $H^0(\mf g/\mf h,[m,-]_G)=\ker(\overline{[m,-]_G:\mf g^0/\mf h^0 \to \mf g^1/\mf h^1}) = 0 = H^1(\mf g/\mf h, \overline{[m,-]_G})$, independent of $m$. It then follows from Theorem \ref{thm:mainthm} that the map \eqref{eq:inclunital} is a local bijection.
\begin{lem}
    The spaces $K:=\ker(\overline{[m,-]_G:\mf g^0/\mf h^0 \to \mf g^1/\mf h^1})$ and $H^1(\mf g/\mf h,\overline{[m,-]_G})$ vanish for an associative algebra $(V,m)$ with unit $1\in V$.
\end{lem}
\begin{proof}
    Note that $(\mf h^\bullet,[m,-]_G)$ is the \emph{normalized} Hochschild complex and the natural map from normalized Hochschild cohomology into Hochschild cohomology is a quasi-isomorphism (see e.g. \cite{lodaybook}). Using the long exact sequence in cohomology associated to the short exact sequence $$ \mf h^\bullet \hookrightarrow \mf g^\bullet \twoheadrightarrow \mf g/\mf h$$ implies that $(\mf g/\mf h,\overline{[m,-]_G})$ is acyclic.
\end{proof}
\begin{rmk}
    Note that this is an incarnation of \cite[Theorem 2.4]{goldmanmillson}, which states that quasi-isomorphic differential graded Lie algebras encode equivalent deformation theories. The difference however, is that our formulation is local rather than formal.
\end{rmk}
Summarizing, Theorem \ref{thm:mainthm} applied to $(\mf g^\bullet =  (V^\ast)^{\otimes(\bullet+1)} \otimes V, 0,[-,-]_G) $, $\mf h^\bullet= ((V/\mb K1)^\ast)^{\otimes(\bullet+1)} \otimes V $ and the Maurer-Cartan element $m$ yields:
\begin{thm}\label{thm:unitvsnon}
    Let $(V,m,1)$ be a unital associative algebra. Then any associative algebra structure $m'$ on $V$ near $m$ is unital. 
\end{thm}
\subsection{Rigidity of Lie algebra morphisms}\label{sec:rigmorphism}
Let $(V,\mu)$ and $(W,\nu)$ be Lie algebras. A Lie algebra morphism $f:(V,\mu) \to (W,\nu)$ is a linear map $f:V\to W$ such that
$$
f(\mu(x,y)) = \nu(f(x),f(y))
$$
for $x,y\in V$. 

There is a differential graded Lie algebra $(\mf g,\partial,[-,-])$ such that its Maurer-Cartan elements are the Lie algebra morphisms $(V,\mu)\to (W,\nu)$.
\begin{lem}[\cite{morphdgla}]
    Let $(V,\mu)$ and $(W,\nu)$ be Lie algebras. The graded vector space $\mf g^\bullet := \wedge^\bullet V^\ast \otimes W$ carries a differential graded Lie algebra structure: the differential is $\partial := d_{CE}$, the Chevalley-Eilenberg differential of $V$, where $W$ is viewed as the trivial $V$-module, and the bracket $[-,-]:= [-,-]_\nu$ is the $\wedge^\bullet V^\ast$-linear extension of $\nu$.

    An element $f\in \mf g^1= V^\ast \otimes W$ is Maurer-Cartan, i.e. $d_{CE}(f) + \frac{1}{2}[f,f]_\nu = 0$, if and only if $f:(V,\mu) \to (W,\nu)$ is a Lie algebra morphism. Moreover, for $X\in \mf g^0 = W$, $f\in \mf g^1 $, we have 
    $$
    f^X = \exp(-\text{ad}_X)\circ f.
    $$
\end{lem}
The gauge equivalences are therefore given by inner automorphisms of $W$. Now, given a morphism $f:(V,\mu)\to (W,\nu)$ of Lie algebras, we obtain the differential $d_{CE} + [f,-]_\nu$ on $\mf g$, which we can describe explicitly.
\begin{lem}
The differential $d_{CE} +[f,-]_\nu$ on $\mf g$ is equal to the Chevalley-Eilenberg differential of $(V,m)$, with values in $W$, where $W$ is seen as $(V,\mu)$-module through $f$. 
\end{lem}
Consequently, we obtain:
\begin{thm}[\cite{cohomgradlie}] \label{thm:rigmorph}
    Let $f:(V,\mu)\to (W,\nu)$ be a morphism of Lie algebras. If
    $$
    H^1_{CE}(V,W) = 0,
    $$
    where $W$ is seen as $(V,\mu)$-module through $f$, then $f$ is rigid. Moreover, for $f'$ near $f$, the elements $w\in W$ such that $\exp(-\text{ad}_w) \circ f' = f$ are smoothly parameterized by an open neighborhood of $0$ in the centralizer of $f(V)\subset W$.
\end{thm}
\subsection{Stability of Lie subalgebras}\label{ssec:stabsubalg}
In this section, we discuss the stability question for Lie subalgebras $W\subset V$: 

\begin{center} Given a Lie algebra $(V,\mu)$, and a subalgebra $W\subset V$, when do all Lie algebra structures $\mu' \in \wedge^2 V^\ast \otimes V$ near $\mu$ admit a subalgebra $W' \subset V$? 
\end{center}

This question was first addressed in \cite{stabsubalgmorph}.

We will follow the procedure as described in Section \ref{sec:stability} to identify the ingredients for Theorem \ref{thm:mainthm}. Let $(V,\mu)$ be a Lie algebra, and consider the graded Lie algebra
$$
(\mf g^\bullet:= \wedge^{\bullet+1} V^\ast \otimes V, 0,[-,-]_{NR}).
$$
As described before, Maurer-Cartan elements of $\mf g$ are precisely $\mu' \in \wedge^2 V^\ast \otimes V$, satisfying the Jacobi identity, i.e. Lie algebra structures on $V$. Now let $W\subset V$ be a Lie subalgebra with respect to $\mu$, i.e. $\mu(W,W) \subset W$. In order to define the graded Lie algebra which controls the deformations of $\mu$ such that $W$ is a subalgebra, we observe that $\mu' \in \wedge^2 V^\ast \otimes V$ preserves $W\subset V$ if and only if $\mu'$ lies in the kernel of the natural map
$$
\wedge^2 V^\ast \otimes V \to \wedge^2 W^\ast \otimes V/W, \,\,\, \mu \mapsto \left. \mu\right|_{\wedge^2 W} \text{ mod } W.
$$
This condition can easily be generalized to the other degrees, so we define 
\begin{equation}\label{eq:subalgstabsubalg}
\mf h^i_{W}:= \ker(\wedge^{i+1} V^ \ast \otimes V \to \wedge^{i+1}W^\ast \otimes V/W)
\end{equation}
Then $\mf h_W^\bullet$ is a graded Lie subalgebra, as can be checked by direct computation.
\begin{lem}
    The graded subspace $\mf h_W^\bullet \subset \mf g^\bullet$ is a graded Lie subalgebra of $\mf g^\bullet$. 
\end{lem}
This puts us directly into the setting of Theorem \ref{thm:mainthm}: for the graded Lie algebra

$(\mf g^\bullet = \wedge^ {\bullet+1} V^\ast \otimes V,0,[-,-]_{NR})$, the graded Lie subalgebra $\mf h$ as in \eqref{eq:subalgstabsubalg} and the Maurer-Cartan element $Q_0 = \mu$, we obtain:
\begin{thm}
If $$H^1(\mf g/\mf h_W, \overline{[\mu,-]_{NR}}) = 0,$$ then for any Lie algebra structure $\mu'$ near $\mu$, there exists a family $I\subset \End(V)$, smoothly parametrized by an open neighborhood of $0\in \ker(\overline{[\mu,-]_{NR}}:\mf g^0/\mf h_W^0 \to \mf g^1/\mf h_W^1)$, such that for $A \in I$, we have $(\mu')^A \in \mf h^1_W$.  
\end{thm}
We now interpret this result. In particular, we 
\begin{itemize}
    \item[i)] interpret the conclusion of the result,
    \item[ii)] give a description of $H^1(\mf g/\mf h,\overline{[\mu,-]_{NR}})$ in terms of Lie algebra cohomology.
\end{itemize}
This interpretation is the content of the next lemma.
\begin{lem} Let $V$ be a vector space and let $ W \subset V$ be a subspace.
\begin{itemize}
    \item[i)] For $\mu' \in \wedge^2 V^\ast \otimes V$, $A \in \End(V)$, we have $(\mu')^A \in \mf h^1_W$ if and only if $\mu'$ preserves the subspace $\exp(-A)(W)$.
    \item[ii)] The complex $(\mf g^\bullet/\mf h^\bullet,\overline{[\mu,-]_{NR}})$ is canonically isomorphic to $(\wedge^\bullet W^\ast\otimes V/W,d_{CE}^\sigma)$, which is the Chevalley-Eilenberg complex of $W$ with coefficients in the representation $V/W$, where $W$ acts via the Lie bracket of $V$. In particular, $\ker(\overline{[m,-]_{NR}}:\mf g^0/\mf h^0_W \to \mf g^1/\mf h^1_W)$ is the space of $V/W$-valued derivations on $(W,\left.\mu\right|_{\wedge^2 W})$.
\end{itemize}
\end{lem}
\begin{proof} 
    \begin{itemize}
        \item []
        \item[i)] Let $\mu'\in \wedge^2V^\ast \otimes V$ and $A \in \End(V)$. Then
        \begin{align*}
                (\mu')^A \in \mf h^1_W &\iff \exp(A)(\mu'(\exp(-A) (W),\exp(-A) (W))) \subset W\\
                & \iff \mu'(\exp(-A)(W),\exp(-A)(W))\subset \exp(-A)(W)\\
                & \iff \text{$\mu'$ preserves $\exp(-A)(W)$.}
        \end{align*}
        \item[ii)] Follows by a direct computation.
    \end{itemize}
\end{proof}
Note that Theorem \ref{thm:mainthm} implies the existence of an $A\in \mf g^0$ of a special form: it lies in a complementary subspace to $\mf h^0\subset \mf g^0$. By picking an appropriate splitting, we can give a geometric interpretation of this fact. Let $C \subset V$ be a complement to $W$, i.e. $V = W \oplus C $. Then there is a natural inclusion $W^\ast \otimes V/W \cong W^\ast \otimes C \hookrightarrow \End(W\oplus C) \cong \End(V)  $. Under this identification, for $A \in W^\ast \otimes C$, $\exp(-A)(W)$ is simply the graph of $-A$ as a subset of $W\oplus C = V$.
So we obtain:
\begin{thm}[\cite{stabsubalgmorph}]\label{thm:stabsubalg}
    Let $(V,\mu)$ be a Lie algebra, and $W\subset V$ a Lie subalgebra. If $$H^2_{CE}(W,V/W) = 0,$$ then $W$ is a stable subalgebra. That is, any Lie algebra structure near $\mu$ admits a family of subalgebras $W'$ near $W$, parameterized by an open neighborhood of $0\in \text{Der}((W,\left.\mu\right|_{\wedge^2 W}),V/W)$.
\end{thm}

\subsection{Stability of morphisms of Lie algebras}\label{ssec:stabmorph}
As a related question, we investigate the stability of a morphism of Lie algebras, which was also considered in \cite{stabsubalgmorph}. Here, we ask the following question:\begin{center}Given a map of Lie algebras $f:(V,\mu) \to (W,\nu)$, when do all Lie algebra structures $\mu' \in \wedge^2V^\ast \otimes V, \nu' \in \wedge^2 W^\ast \otimes W$ near $\mu$ and $\nu$ respectively admit a Lie algebra morphism $f':(V,\mu') \to (W,\nu')$?  \end{center}

We first apply Theorem \ref{thm:mainthm} in a naive way, which, for $\mu', \nu'$ near $\mu,\nu$ as above, will give a criterion for the existence of endomorphisms $A \in \End(V), B\in \End(W)$ such that 
$$
\exp(-B)\circ f \circ \exp(A)
$$
is a morphism of Lie algebras $(V,\mu') \to (W,\nu')$. While this does give a condition for the existence of a morphism, it only detects morphisms of the same rank as $f$, which is a consequence of the way the gauge equivalences act in the graded Lie algebra controlling the deformations of $\mu$ and $\nu$. 

As we are deforming two Lie algebra structures simultaneously and independently, the natural graded Lie algebra to consider is 
$$
(\mf g^\bullet := \wedge^{\bullet+1} V^\ast \otimes V \oplus \wedge^ {\bullet+1} W^\ast \otimes W,0,[-,-]_{NR}).
$$
Here the summands are equipped with the natural Nijenhuis-Richardson bracket, while all cross-terms vanish. The Maurer-Cartan elements are pairs $(\mu',\nu')\in \wedge^2 V^ \ast \otimes V \oplus \wedge^2 W^\ast \otimes W$ such that $\mu'$ and $\nu'$ are Lie algebra structures on $V,W$ respectively. Now pick $(\mu,\nu) \in \mf g^1$ to be a Lie algebra structure, and let $f:(V,\mu) \to (W,\nu)$ be a morphism of Lie algebras. 

Following the procedure, we want to identify a differential graded Lie subalgebra $\mf h^\bullet \subset \mf g^\bullet$, such that the Maurer-Cartan elements are those pairs of Lie algebra structures on $V$ and $W$, such that $f$ is a morphism between them. Note that for $(\mu',\nu')\in \mf g^1$, $f$ is a morphism $(V,\mu')\to(W,\nu')$ if and only if
$$
f\circ \mu' - \nu'\circ f\wedge f = 0.
$$
Again, this condition is easily generalized to higher degrees.
Define the map
$$
F^\bullet:\wedge^{\bullet+1} V^\ast \otimes V \oplus \wedge^{\bullet +1} W^\ast \otimes W \to \wedge^{\bullet +1} V^\ast \otimes W, \,\,\,(\alpha,\beta) \mapsto f\circ \alpha - \beta\circ f^{\wedge(\bullet+1)}.
$$
Then we set
\begin{equation}\label{eq:subalgmorphism}
\mf h_f^\bullet := \ker(F^\bullet).
\end{equation}
It is easy to check that $\mf h$ is a graded Lie subalgebra. We are now directly in the framework of Theorem \ref{thm:mainthm}: for the graded Lie algebra $(\mf g^\bullet:= \wedge^{\bullet+1} V^\ast \otimes V \oplus \wedge^{\bullet +1} W^\ast \otimes W,0,[-,-]_{NR})$, the graded Lie subalgebra $\mf h_f^\bullet = \ker(F)$, and the Maurer-Cartan element $(\mu,\nu) \in \mf g^1$, we obtain:
\begin{thm}\label{thm:stabmorph1}
    Let $f:(V,\mu) \to (W,\nu)$ be a morphism of Lie algebras. Assume that $$H^1(\mf g/\mf h_f,\overline{[(\mu,\nu),-]_{NR}}) = 0.$$ Then for any pair of Lie algebra structures $(\mu',\nu')$ near $(\mu,\nu)$ there are endomorphisms $(A,B) \in \End(V)\oplus \End(W)$, parametrized by an open neighborhood of $0\in\text{Der}(V,\mu) \oplus \text{Der}(W,\nu)$, such that $f$ is a morphism of Lie algebras $(V,(\mu')^A) \to (W,(\nu')^B)$.
\end{thm}
As mentioned at the beginning of this section, the conclusion implies that 
$$
\exp(-B)\circ f \circ \exp(A): (V,\mu') \to (W,\nu')
$$
is a morphism of Lie algebras.

Note that this result only detects Lie algebra morphisms which are conjugate to $f$ by automorphisms of $V$ and $W$. Ideally, we would like a result for all nearby morphisms. For that we recall the following lemma, characterizing Lie algebra morphisms.
\begin{lem}
    Let $f:(V,\mu)\to (W,\nu)$ be a linear map between Lie algebras. Then $f$ is a Lie algebra morphism if and only if the graph of $f$ is a Lie subalgebra of $(V\oplus W, \mu\oplus \nu)$.
\end{lem}
Using this observation, we will give a criterion for stability of the graph of a Lie algebra morphism $f:(V,\mu) \to (W,\nu)$ as subalgebra of $(V\oplus W,\mu\oplus \nu)$. Note that any subspace close to the graph of $f$ is the graph of a linear map $f':V\to W$, due to the transversality to $W\subset V\oplus W$. However, not every Lie algebra structure $M$ on $V\oplus W$ is of the form $\mu' \oplus \nu'$ for Lie algebra structures $\mu',\nu'$ on $V,W$ respectively. Therefore, applying Theorem \ref{thm:stabsubalg} to the Lie algebra $(V\oplus W,\mu\oplus \nu)$ and the subalgebra $\text{graph}(f)$, we obtain:
\begin{thm} \label{thm:stabmorph2}
   Let $f:(V,\mu) \to (W,\nu)$ be a morphism of Lie algebras. If $$H^2_{CE}\left(\text{graph}(f),\frac{V\oplus W}{\text{graph}(f)}\right) = 0,$$ then for any Lie algebra structure $M$ on $V \oplus W$ there are linear maps $f':V\to W$ near $f$, parameterized by an open neighborhood of $0\in \text{Der}\left((\text{graph}(f),\left.\mu\right|_{\wedge^2 \text{graph}(f)}),\frac{V\oplus W}{\text{graph}(f)}\right)$ such that $\text{graph}(f') \subset (V\oplus W,M)$ is a Lie subalgebra. In particular, when $M = \mu' \oplus \nu'$, $f':(V,\mu') \to (W,\nu')$ is a Lie algebra morphism.
\end{thm}
\begin{proof}
     By Theorem \ref{thm:stabsubalg} applied to the Lie algebra $(V\oplus W, \mu \oplus \nu)$ and the Lie subalgebra $\text{graph}(f)\subset V\oplus W$, for every Lie algebra structure $M$ on $V \oplus W$, there is a subspace $T \subset V\oplus W$ near $\text{graph}(f)$ which is a Lie subalgebra of $(V\oplus W, M)$. As any subspace close enough to $\text{graph}(f)$ can be written as the graph of a linear map $f':V\to W$, the result follows. In particular, for decomposable Lie algebra structures $M = \mu' \oplus \nu'$, this implies that $f':(V,\mu') \to (W,\nu')$ is a Lie algebra morphism.
\end{proof}
\begin{rmk}
    Using the identifications $V \cong \text{graph}(f), v\mapsto (v,f(v))$ and $W\cong \frac{V\oplus W}{\text{graph}(f)},   w \mapsto (0,w) \mod \text{graph}(f)$, there is an isomorphism
    $$
    C^\bullet_{CE}\left(\text{graph}(f),\frac{V\oplus W}{\text{graph}(f)}\right) \cong C_{CE}^\bullet(V,W),
    $$
    where $W$ is seen as a $V$-module through $f$. This is the same cohomological hypothesis as \cite{crainic2013survey} and \cite{stabsubalgmorph}, but our result is stronger: under the same assumptions, the morphism is stable under a larger class of deformations.
\end{rmk}
\begin{rmk}
There is a relation between Theorems \ref{thm:stabmorph1} and \ref{thm:stabmorph2}: given a morphism of Lie algebras $f:(V,\mu) \to (W,\nu)$, deformations of $\mu$ and $\nu$ induce deformations of the Lie algebra $(V\oplus W, \mu\oplus \nu)$. This is reflected in a graded Lie algebra morphism
$$
(\mf g_1^\bullet := \wedge^{\bullet+1} V^\ast \otimes V \oplus \wedge^{\bullet +1} W^\ast \oplus W,0,[-,-]_{NR}) \to (\mf g_2^\bullet := \wedge^{\bullet+1}(V\oplus W)^\ast \otimes (V\oplus W),0, [-,-]_{NR}).
$$
mapping $(\mu,\nu)$ to $\mu\oplus \nu$.

Moreover, $\mf h_f^\bullet$ as in \eqref{eq:subalgmorphism} is mapped to $\mf h_{graph(f)}^\bullet$ as in \eqref{eq:subalgstabsubalg}. Therefore, there is an induced map on the cohomological obstructions
\begin{equation}\label{eq:cohomindf}
H^1(\mf g_1/\mf h_{f}, \overline{[(\mu,\nu),-]_{NR}}) \to H^1(\mf g_2/\mf h_{\text{graph}(f)},\overline{[\mu\oplus\nu,-]_{NR}}). 
\end{equation}
Although stability of the $\text{graph}(f)$ as subalgebra of $(V\oplus W,\mu\oplus \nu)$ implies stability of $f$ as a morphism $f:(V,\mu)\to (W,\nu)$, the vanishing of the right hand side of \eqref{eq:cohomindf} does not imply the vanishing of the left hand side, as the vanishing of the left hand side only detects morphisms of the same rank as $f$.
\end{rmk}
\subsection{Stability of mapping into a fixed subalgebra}\label{ssec:mapintosubalg}
In this section, we investigate the stability of a Lie algebra morphism $f:(V,\mu) \to (W,\nu)$ mapping into a given subalgebra $U\subset W$:
\begin{center}
    Given a morphism $f:(V,\mu) \to (W,\nu)$ such that $f(V) \subset U$, where $U\subset W$ is a Lie subalgebra, when do all nearby morphisms $f'$ satisfy $f'(V) \subset U'$ for some nearby Lie subalgebra $U'\subset W$?
\end{center}
We follow the procedure outlined in Section \ref{sec:stability} to identify the ingredients for Theorem \ref{thm:mainthm}. Let $f:(V,\mu) \to (W,\nu)$ be a morphism of Lie algebras, and consider the differential graded Lie algebra 
$$
(\mf g^\bullet = \wedge^\bullet V^\ast \otimes W, d_{CE},[-,-]_\nu).
$$
Recall that Maurer-Cartan elements are precisely Lie algebra morphisms $f:(V,\mu) \to (W,\nu)$. Now assume that $U\subset W$ is a Lie subalgebra, and that $f(V)\subset U$. To define the differential graded Lie subalgebra which encodes deformations $f'$ of $f$ such that $f'(V)\subset U$, we note that $f'(V)\subset U$ if and only if $f' \in V^\ast \otimes U \subset V^\ast \otimes W$. This condition has a clear generalization to other degrees, so define
\begin{equation}\label{eq:subalgmorphmap}
\mf h^\bullet_{U} := \wedge^\bullet V^\ast \otimes U.
\end{equation}
Then it is easy to see that $\mf h^\bullet_{U}$ is a differential graded Lie subalgebra of $(\mf g^\bullet = \wedge^\bullet V^\ast \otimes W ,d_{CE},[-,-]_{\nu})$.

Applying Theorem \ref{thm:mainthm} to $(\mf g^\bullet = \wedge^\bullet V^\ast \otimes W ,d_{CE},[-,-]_{\nu})$, $\mf h^\bullet = \wedge^\bullet V^\ast \otimes U$ and the Maurer-Cartan element $Q_0 = f$, we obtain:
\begin{thm}\label{thm:morphsubalgmap}
If $$H^1(\mf g/\mf h_{U}, \overline{d_{CE} +[f,-]_{\nu}}) = 0,$$ then for any Lie algebra morphism $f':(V,\mu) \to (W,\nu)$ near $f$, there exists a family $I\subset W$, smoothly parameterized by an open neighborhood of $0 \in \ker(\overline{d_{CE} +[f,-]_{\nu}}:\mf g^0/\mf h^0_{U} \to  \mf g^1/\mf h^1_{U})$, such that for $w \in I$, we have $(f')^w \in \mf h^1_{U}$.
\end{thm}
In the following lemma we unpack the content of Theorem \ref{thm:morphsubalgmap}, of which we leave the proof to the reader.
\begin{lem} Let $f:(V,\mu)\to (W,\nu)$ be a morphism of Lie algebras. 
    \begin{itemize}
        \item[-] For $f':V\to W$, $w\in W$, we have $(f')^w\in \mf h^1_U$ if and only if $f'(V)\subset \exp(\text{ad}_w)(U)$.
        \item[-] The complex $(\mf g/\mf h_U, \overline{d_{CE} + [f,-]_\nu})$ is isomorphic to the Chevalley-Eilenberg complex of $(V,\mu)$ with values in $W/U$, where $W/U$ is a $(V,\mu)$-module through $f$. In particular, 
        $\ker(\overline{d_{CE} +[f,-]_\nu}:\mf g^0/\mf h^0_U \to \mf g^1/\mf h^1_U) = \{[w] \in W/U \mid \nu(w,f(V)) \in U\}. $
    \end{itemize}
\end{lem}
We therefore obtain:
\begin{thm}\label{thm:morphsubalgconcr} Let $f:(V,\mu) \to (W,\nu)$ be a morphism of Lie algebras, and assume that $f(V) \subset U$, where $U\subset W$ is a Lie subalgebra. If 
$$
H^1_{CE}(V,W/U) = 0,
$$
then for any Lie algebra morphism $f':(V,\mu) \to (W,\nu)$ near $f$, there exists a family $I\subset W$, parameterized by an open neighborhood of $0\in\{[w] \in W/U \mid \nu(w,f(V)) \in U\} $, such that for $w\in I$, we have $f'(V) \subset \exp(\text{ad}_w)(U)$. 
\end{thm}
For the rest of this section, let $W = \End(C)$ for some vector space $C$. By picking different subalgebras of $\End(C)$, we will give criteria for the action to preserve a geometric structure up to equivalence. 
\begin{exmp}\label{ex:geostr} Let $\rho:(V,\mu)\to \End(C)$ be a Lie algebra morphism, i.e. a representation of $(V,\mu)$. Assume that $f(V)\subset U$, where $U\subset W$ is a Lie subalgebra.
\begin{itemize}
    \item [-] Let $g \in S^2(C^\ast)$ be a non-degenerate symmetric bilinear form. Taking $$U = \mf {so}(C,g) = \{A \in \End(C) \mid \forall x,y \in C: g(Ax,y) + g(x,Ay) = 0\},$$
    Theorem \ref{thm:morphsubalgconcr} implies that if 
    $$
    H^1_{CE}(V, W/U) = H^1_{CE}(V,S^2(C^\ast)) =  0, 
    $$
    where $(V,\mu)$ acts on $S^2(C^\ast)$ by the natural extension, then for any representation $\rho'$ near $\rho$, there exists an $A \in \End(V)$, such that $\rho'(V) \subset \exp(A)\mf{so}(C,g)\exp(-A) = \mf {so}(C,\exp(A)\cdot g)$, where $(\exp(A)\cdot g)(x,y) = g(\exp(-A)x,\exp(-A)y)$.
    \item[-] Let $\omega \in \wedge^2 C^\ast$ be a non-degenerate skew-symmetric bilinear form. Taking $$ U= \mf {sp}(C,\omega) = \{A \in \End(C) \mid \forall x,y \in C: \omega(Ax,y) + \omega(x,Ay) = 0\},$$
    Theorem \ref{thm:morphsubalgconcr} implies that if 
    $$
    H^1_{CE}(V,W/U) = H^1_{CE}(V,\wedge^2 C^\ast) = 0,
    $$
    where $(V,\mu)$ acts on $\wedge^2 C^\ast$ by the natural extension, then for any representation $\rho'$ near $\rho$, there exists an $A \in \End(V)$, such that $\rho'(V) \subset \exp(A)\mf{sp}(C,\omega)\exp(-A) = \mf{sp}(C,\exp(A)\cdot \omega). $
\end{itemize}
    
\end{exmp}
\footnotesize
\bibliography{bib}
\bibliographystyle{alpha}
\Addresses
\end{document}